\documentclass[reqno]{amsart}
\usepackage{amssymb, amsmath}
\usepackage{dsfont}
\usepackage{natbib}
\usepackage[pdftex,plainpages=false,colorlinks,hyperindex,bookmarksopen,linkcolor=red,citecolor=blue,urlcolor=blue]{hyperref}

\bibpunct{[}{]}{;}{n}{,}{,}

\newtheorem{te}{Theorem}[section]

\newtheorem{os}[te]{Remark}

\newtheorem{prop}[te]{Proposition}

\newtheorem{coro}[te]{Corollary}

\numberwithin{equation}{section}

\allowdisplaybreaks

\def \l { \left( }
\def \r {\right) }
\def \ll { \left\lbrace }
\def \rr { \right\rbrace }

\begin{document}

	\title[Counting processes with Bern\v{s}tein intertimes and random jumps]{Counting processes with Bern\v{s}tein intertimes and random jumps}
	\author{Enzo Orsingher}
	\author{Bruno Toaldo} 
	\address{Department of Statistical Sciences, Sapienza University of Rome}
	\email{enzo.orsingher@uniroma1.it} \email{bruno.toaldo@uniroma1.it}
	\date{\today}

		\begin{abstract}
We consider here point processes $N^f(t)$, $t>0$, with independent increments and integer-valued jumps whose distribution is expressed in terms of Bern\v{s}tein functions $f$ with L\'evy measure $\nu$. We obtain the general expression of the probability generating functions $G^f$ of $N^f$, the equations governing the state probabilities $p_k^f$ of $N^f$, and their corresponding explicit forms. We also give the distribution of the first-passage times $T_k^f$ of $N^f$, and the related governing equation. We study in detail the cases of the fractional Poisson process, the relativistic Poisson process and the gamma-Poisson process whose state probabilities have the form of a negative binomial. The distribution of the times $\tau_j^{l_j}$ of jumps with height $l_j$ ($\sum_{j=1}^rl_j = k$) under the condition $N(t) = k$ for all these special processes is investigated in detail.
		\end{abstract}
	
	\maketitle

\section{Introduction}
In this paper we consider a class of point processes with stationary independent integer-valued increments of arbitrary range. These processes can be regarded as generalizations of the Poisson process where jumps can take any positive value. Furthermore, we shall show that these processes $N^f(t)$, $t>0$, can also be viewed as time-changed Poisson processes $N \l \, H^f(t) \, \r$ where $H^f(t)$ are subordinators associated with the Bern\v{s}tein function $f$ and independent from the homogeneous Poisson process $N$ with rate $\lambda >0$. The probabilistic behaviour of the processes $\mathcal{P}^f(t)$, with related counting processes $N^f(t)$, is described by the following properties.
\begin{enumerate}
\item[i)] $\mathcal{P}^f(t)$ has independent and stationary increments;
\item[ii)] \begin{align}
\Pr \ll N^f[t, t+dt) = k \rr \, = \, \begin{cases} dt \frac{\lambda^k}{k!} \int_0^\infty e^{-\lambda s} s^k \nu(ds)+o(dt), \qquad &k \geq 1, \\ 
1-dt \int_0^\infty \l 1-e^{-\lambda s} \r \nu(ds) + o(dt), & k=0,
\end{cases}
\label{11}
\end{align}
where
\begin{align}
f(\lambda) \, = \,  \int_0^\infty \l 1-e^{-\lambda s} \r \nu(ds)
\label{12}
\end{align}
is the integral representation of the Bern\v{s}tein functions. 
\end{enumerate}
The Bern\v{s}tein functions are $C^\infty$, non-negative and such that $(-1)^k \frac{d^k}{dx^k} f(x) \leq 0$, $k \geq 1$ (see, for example, \citet{librobern}). By $\nu$ we denote a non-negative L\'evy measure on the positive half-line such that
\begin{align}
\int_0^\infty \l s \wedge 1 \r \nu(ds) \, < \, \infty.
\end{align}
We often speak of $N^f(t)$, $t>0$, as generalized Poisson processes performing integer-valued jumps of arbitrary height. 

These processes can be used to model many different concrete and real phenomena. For example, if we consider the car accidents in the time interval $[0,t)$, the number of injured people in each crash can take any positive number. The number of clients appearing at any commercial center and arriving on different transport vehicles can also be modeled by a suitable counting process $N^f(t)$, $t>0$. Analogously, in floods or earthquakes, the number of destroyed buildings in each event can be clearly of arbitrary magnitude and thus can be represented by $N^f(t)$, $t>0$, with suitably chosen Bern\v{s}tein function $f$ and L\'evy measure $\nu$. 

The subordinators $H^f$ have Laplace transforms
\begin{align}
\mathds{E}e^{-\mu H^f(t)} \, = \, e^{-tf(\mu)} \, = \, e^{-t \int_0^\infty \l 1-e^{-s\mu} \r \nu(ds)}.
\label{15}
\end{align}
We observe that for
\begin{align}
\nu(ds) \, = \, \frac{\alpha s^{-\alpha -1}}{\Gamma(1-\alpha)} ds, \qquad \alpha \in (0,1),
\end{align}
we obtain the space-fractional Poisson process studied in \citet{orspolspl}, where $f(\mu) = \mu^\alpha$, $\alpha \in (0,1)$.
In this case, the subordinator corresponding to the space-fractional Poisson is a positively skewed stable process of order $\alpha$. If the L\'evy measure is the Dirac point mass at one, then the corresponding subordinated Poisson process is
\begin{equation}
N_1 \l N_2 (t) \r, \qquad t>0,
\end{equation}
where $N_i$, $i=1,2$, are independent homogeneous Poisson processes with rates $\lambda_i >0$. Such a process was investigated in \citet{orspoljsp} and recently also in \citet{dicre}.

The state probabilities $p_k^f(t) = \Pr \ll N^f(t) = k \rr$ are governed by difference-differential equations of the form
\begin{align}
\frac{d}{dt} p_k^f(t) \, = \, -f(\lambda) p_k^f(t) +\sum_{m=1}^k \frac{\lambda^m}{m!} p_{k-m}^f(t) \int_0^\infty e^{-s\lambda} s^m \nu(ds), \qquad k \geq 0, t>0,
\label{unosette}
\end{align}
with the usual initial conditions.
From \eqref{unosette} we extract the probability generating function $G^f(u, t)$ of $N^f(t)$ as \begin{align}
G^f(u, t) \, = \, e^{-tf \l \lambda (1-u) \r} \, = \, e^{-t \int_0^\infty \l 1-e^{-s\lambda(1-u)} \r} \nu(ds).
\end{align}
We prove also that
\begin{align}
\mathds{E}u^{N \l H^f(t) \r} \, = \, e^{-tf \l \lambda(1-u) \r}
\end{align}
and thus we show that
\begin{align}
N^f(t) \, \stackrel{\textrm{law}}{=} \, N \l H^f(t) \r.
\end{align}
By means of the shift operator $B^mp_k^f(t) = p_{k-m}^f(t)$, $0 \leq m \leq k$, we can rewrite equation \eqref{unosette} as
\begin{align}
\frac{d}{dt} p_k^f(t) \, = \, -f \l \lambda \l I-B \r \r p_k^f(t), \qquad t>0, k \geq 0,
\end{align}
which for $f(x) = x^\alpha$ coincides with the equation (2.4) of \citet{orspolspl}.
We also present a further representation of the generalized Poisson process $\mathcal{P}^f(t)$, $t>0$, as the scale limit of a continuous-time random walk with steps $X_j$ having distribution
\begin{align}
\Pr \ll X_j = k \rr \, = \, \frac{1}{u(n)} \int_0^\infty \Pr \ll N(s) = k \rr \,  \nu(ds), \qquad k \geq n \in \mathbb{N},
\label{unoundici}
\end{align}
where
\begin{align}
u(n) \, = \, \int_0^\infty \Pr \ll N(s) \geq n \rr \, \nu(ds).
\end{align}
For example, for the space-fractional Poisson process the distribution \eqref{unoundici} becomes
\begin{align}
\Pr \ll X_j = k \rr \, = \, \frac{\Gamma(k-\alpha) \big/ k!}{\sum_{j=n}^\infty \Gamma(j-\alpha)\big/ j!} , \qquad k \geq n.
\end{align}
We also consider the hitting-times
\begin{align}
T_k^f \, = \, \inf \ll t \geq 0 : N^f(t) \geq k \rr
\end{align}
and we show that
\begin{align}
\Pr \ll T_k^f \in ds \rr \big/ ds \, = \, -\frac{d}{ds} \sum_{l=0}^{k-1} \frac{(-\lambda)^l}{l!} \frac{d^l}{d\lambda^l} e^{-sf(\lambda)}.
\label{uno15}
\end{align}
We note that for $f(\lambda) = \lambda$ (case of the homogeneous Poisson process) formula \eqref{uno15} yields the Erlang distribution
\begin{equation}
\Pr \ll T_k \in ds \rr \, = \, \lambda^k e^{-\lambda s} \frac{s^{k-1}}{(k-1)!} ds, \qquad k \geq 1, s>0.
\end{equation}
The last part of the paper is devoted to three special cases, that is,
\begin{align}
f(\mu) \, = \, \begin{cases} \mu^\alpha, \, \alpha \in (0,1) &\textrm{ (space-fractional Poisson process),} \\
(\mu + \theta)^\alpha - \theta^\alpha, \, \alpha \in (0,1) &\textrm{ (tempered Poisson process),} \\
\log (1+\mu)  \,  &\textrm{ (negative binomial process).}
 \end{cases}
\end{align}
We obtain explicitly the probability distribution $p_k^f(t)$, $k \geq 0$, in the three cases above. Furthermore, we are able to obtain the conditional distributions, for $0<t_1 < \dots < t_r < t$,
\begin{align}
\Pr \ll \bigcap_{j=1}^r \ll \tau_j^{l_j} \in dt_j \rr \bigg| N^f(t) = k \rr
\end{align}
(where $\tau_j^{l_j}$ are the instants of occurence of the $j$-th Poisson event with size $l_j$) for $f(\mu) = \mu^\alpha$, $f(\mu) = \log(1+\mu)$. The tempered Poisson process has finite moments (unlike the space-fractional Poisson process) as well as the negative binomial of which many particular distributions can be explicitely evaluated.

The Poisson process and the negative binomial processes have been generalized in many directions (see, for example, \citet{beghinbin, brix, cahoy, dicre, vella}). The processes analyzed here include some processes which have appeared recently in the literature, but not the time-fractional Poisson process (which is a renewal process with non-independent increments, see \citet{kreer, kumar, laskin, meerpoisson}).

\section{General results}
We now examine in detail the main properties of the process $N^f(t)$, $t>0$, with independent increments outlined in the introduction. Our first result is the difference-differential equations governing their state probabilities
\begin{align}
p_k^f(t) \, = \, \Pr \ll N^f(t) = k \rr, \qquad k \geq 0.
\end{align}
\begin{te}
The probabilities $p_k^f(t) = \Pr \ll N^f(t) = k \rr$, $k \geq 0$, are solutions to the equation
\begin{align}
\frac{d}{dt}p_k^f(t) \, = \, -f(\lambda) p_k^f(t) + \sum_{m=1}^k \frac{\lambda^m}{m!} p_{k-m}^f (t) \int_0^\infty e^{-s\lambda} s^m \nu (ds), \qquad k \geq 0, t>0,
\label{21}
\end{align}
with initial condition
\begin{equation}
p_k^f(0) \, = \, 
\begin{cases}
1, \qquad & k = 0\\
0, & k \geq 1.
\end{cases}
\end{equation}
The p.g.f. $G^f(u, t) = \mathds{E}u^{N^f(t)}$, $|u|<1$, satisfies the linear, homogeneous equation
\begin{align}
\begin{cases}
\frac{\partial}{\partial t} G^f(u, t) \, = \, -f \l \lambda (1-u) \r G^f(u, t) \\
G^f(u, 0) \, = \, 1,
\end{cases}
\end{align}
and has the form
\begin{equation}
G^f(u, t) \, = \, e^{-tf \l \lambda (1-u) \r}.
\label{212}
\end{equation}
\end{te}
\begin{proof}
Since $N^f(t)$ has independent increments and the distribution of jumps is given by \eqref{11} we can write
\begin{align}
p_k^f(t+dt) \, = \, & \Pr \ll N^f[t+dt) = k \rr \, = \, \Pr \ll \bigcup_{j=0}^k \ll N^f(t) = j, N^f[t, t+dt) = k-j \rr \rr \notag \\
= \, &  \sum_{j=0}^{k-1} \Pr \ll N^f(t) = j \rr dt \frac{\lambda^{k-j}}{(k-j)!} \int_0^\infty e^{-\lambda s} s^{k-j} \nu(ds) \notag \\
& + \Pr \ll N^f(t) =k \rr \l 1-dt \int_0^\infty \l 1-e^{-\lambda s}  \r \nu(ds) \r.
\end{align}
A simple expansion permits us to obtain, in the limit, equation \eqref{21}. From equation \eqref{21} we have that
\begin{align}
\frac{\partial}{\partial t} G^f(u, t) \, = \, & \sum_{k=0}^\infty u^k \frac{d}{d t} p_k^f(t) \notag \\
 = \, & -f(\lambda) \sum_{k=0}^\infty u^k p_k^f(t) + \sum_{k=1}^\infty u^k \sum_{m=1}^k \frac{\lambda^m}{m!} p_{k-m}^f(t) \int_0^\infty e^{-s\lambda} s^m \nu(ds) \notag \\
 = \,  & -f(\lambda) G^f(u, t) + \sum_{m=1}^\infty \frac{\lambda^m}{m!} \int_0^\infty e^{-s\lambda }s^m \nu(ds) \sum_{k=m}^\infty u^k p_{k-m}^f(t) \notag \\
 = \, & -f(\lambda) G^f(u, t) + G^f(u, t) \int_0^\infty \l e^{-s\lambda(1-u)} - e^{-s\lambda} \r \nu(ds) \notag \\
 = \, & -G^f(u, t) \int_0^\infty \l 1-e^{-s\lambda(1-u)} \r \nu(ds) \notag \\
 = \, & -G^f(u, t) \, f(\lambda(1-u)).
\end{align}
In the last step, we take into account the representation \eqref{12} of the Bern\v{s}tein functions.
\end{proof}
\begin{os}
The appearance of $p_{k-j}(t)$, $k \geq j \geq 2$, in \eqref{21} makes the master equation of the state probabilities $p_k^f(t)$, substantially different from the case of the classical Poisson process. This fact is related to the possibility of jumps of arbitrary height. We also observe that
\begin{align}
N^f(t) \, \stackrel{\textrm{\normalfont law}}{=} \, N \l H^f(t) \r
\label{26}
\end{align}
where $H^f$ is the subordinator with Laplace transform \eqref{15}. This can be ascertained by evaluating the p.g.f. of $N \l H^f(t) \r$, $t>0$, as follows
\begin{align}
\mathds{E}u^{N \l H^f(t) \r} \, = \, & \sum_{k=0}^\infty u^k \int_0^\infty \Pr \ll N(s) = k \rr \Pr \ll H^f(t) \in ds \rr \notag \\
= \, & \int_0^\infty e^{-s\lambda(1-u)}\Pr \ll H^f(t) \in ds \rr \notag \\
= \, &G^f(u, t).
\end{align}
\end{os}
In view of \eqref{26} we can write the distribution of $N^f(t)$ as
\begin{align}
\Pr \ll N^f(t) = k \rr \, = \, & \Pr \ll N \l H^f(t) \r = k \rr \, = \, \int_0^\infty e^{-\lambda s} \frac{(\lambda s)^k}{k!} \Pr \ll H^f(t) \in ds \rr \notag \\
= \, & \frac{(-1)^k}{k!} \frac{d^k}{d u^k} \int_0^\infty e^{-\lambda s u} \Pr \ll H^f(t) \in ds \rr \bigg|_{u=1} \notag \\
= \, & \frac{(-1)^k}{k!} \frac{d^k}{d u^k} e^{-tf(\lambda u)} \bigg|_{u=1}, \qquad k \geq 0.
\label{219}
\end{align}
\begin{os}
The equation \eqref{21} can alternatively be written as
\begin{align}
\frac{d}{dt} p_k^f(t)  \, = \, -f \l \lambda \l I-B \r \r p_k^f(t), \qquad t>0, k \geq 0,
\end{align}
where $B$ is the shift operator such that $Bp_k^f(t) = p_{k-1}^f(t)$. This can be shown as follows:
\begin{align}
& -f \l \lambda (I-B) \r p_k^f(t) \notag \\
 = \, & -\int_0^\infty \l 1-e^{-\lambda  s(I-B)} \r \nu(ds) \, p_k^f(t) \notag \\
= \, & -  \int_0^\infty \l 1-e^{-\lambda s} \sum_{m=0}^\infty \frac{\l \lambda s B \r^m}{m!}  \r \nu(ds) \, p_k^f(t) \notag \\
= \, &-\int_0^\infty \l p_k^f(t) - e^{-\lambda s} \sum_{m=0}^k \frac{(\lambda s)^m}{m!} p_{k-m}^f(t) \r \nu(ds) \notag \\
= \, & - \int_0^\infty \l 1-e^{-\lambda s} \r \nu(ds) p_k^f(t) + \sum_{m=1}^k \frac{\lambda^m}{m!} p_{k-m}^f(t) \int_0^\infty e^{-\lambda s} s^m \nu(ds) \notag \\
= \, & - f(\lambda) p_k^f(t) + \sum_{m=1}^k \frac{\lambda^m}{m!} p_{k-m}^f(t) \int_0^\infty e^{-\lambda s} s^m \nu(ds).
\label{28}
\end{align}
Clearly \eqref{28} coincides with the right-hand member of \eqref{21}.
\end{os}
A further representation of $N^f(t)$, $t>0$, can be obtained as the limit of a suitable compound Poisson process.
\begin{te}
Let
\begin{equation}
u(n) \, = \, \int_0^\infty \Pr \ll N(s) \geq n \rr \, \nu(ds), \qquad n \in \mathbb{N},
\end{equation}
where $N(s)$, $s>0$, is a homogeneous Poisson process with rate $\lambda >0$. The compound Poisson process
\begin{align}
Z_n(t) \, = \, \sum_{j=1}^{N \l \frac{t}{\lambda} \, u(n)  \r } X_j, \qquad t>0,
\end{align}
where $X_j$, $j =1, 2,  \dots  $, are discrete i.i.d. r.v.'s with probability law
\begin{align}
\Pr \ll X_j = k \rr \, = \, \frac{1}{u(n)} \int_0^\infty \Pr \ll N(s) = k \rr \,  \, \nu(ds) , \qquad  k \geq n \in \mathbb{N}, \forall j = 1,2, \dots,
\label{211}
\end{align}
converges in distribution to the subordinated Poisson process $N^f(t)$ as $n \to 0$. In other words,
\begin{align}
N^f(t) \, \stackrel{\textrm{\normalfont law}}{=} \, N \l H^f(t) \r \, \stackrel{\textrm{\normalfont law}}{=} \, \lim_{n \to 0} Z_n(t).
\end{align}
\end{te}
\begin{proof}
The p.g.f. of $Z_n(t)$ writes
\begin{align}
\mathds{E}u^{Z_n(t)} \, = \, & e^{-tu(n) \l 1-\mathds{E}u^X \r} \notag \\
= \, & \exp \ll -t u(n) \sum_{k=n}^\infty \l 1-u^k \r \Pr \ll X = k \rr \rr \notag \\
\stackrel{\eqref{211}}{=} \, & \exp \ll -t u(n) \sum_{k=n}^\infty \l 1-u^k \r \frac{1}{u(n)} \int_0^\infty \Pr \ll N(s) = k \rr   \, \nu(ds) \rr \notag \\
= \, & \exp \ll -t \int_0^\infty \sum_{k=n}^\infty \l 1-u^k \r \Pr \ll N(s) = k \rr \nu(ds)  \rr.
\label{asterisco}
\end{align}
By taking the limit for $n \to 0$ of \eqref{asterisco} we have that
\begin{align}
\lim_{n \to 0} \mathds{E}u^{Z_n(t)} \, = \, & \exp \ll -t \int_0^\infty \sum_{k=0}^\infty \l 1-u^k \r \Pr \ll N(s) = k \rr \nu(ds) \rr \notag \\
= \, & \exp \ll-t \int_0^\infty \l 1-e^{-\lambda s(1-u)} \r \nu(ds) \rr \notag \\
= \, & e^{-tf(\lambda(1-u))}.
\end{align}
\end{proof}
\begin{os}
If we take into account processes whose state probabilities satisfy the time-fractional equation
\begin{align}
\frac{d^\nu}{dt^\nu} p_k^f(t) \, = \,  -f(\lambda) p_k^f(t) + \sum_{m=1}^k \frac{\lambda^m}{m!} p_{k-m}^f (t) \int_0^\infty e^{-s\lambda} s^m \nu (ds), \qquad k \geq 0, t>0,
\label{fame}
\end{align}
for $\nu \in (0,1)$ the corresponding p.g.f. has the form
\begin{align}
G^f_\nu(u, t) \, = \, E_{\nu,1} \l -t^\nu \int_0^\infty \l 1-e^{-s\lambda(1-u)} \r \nu(ds) \r
\label{due20}
\end{align}
where $E_{\nu,1}(x)$ is the Mittag-Leffler function and the fractional derivative appearing in \eqref{fame} must be understood in the Caputo sense. For the space-fractional Poisson process $f(\lambda) = \lambda^\alpha$, $0 < \alpha < 1$, the distribution of the process related to \eqref{due20} is explicitly given by formula (2.29) of \citet{orspolspl}. The processes whose distribution is governed by \eqref{fame} admits the following representation
\begin{align}
N \l H^f \l L^\nu (t) \r \r, \qquad t>0,
\end{align}
where $L^\nu$ and the stable subordinator $H^\nu$ are related by
\begin{align}
\Pr \ll L^\nu(t) > x \rr \, = \, \Pr \ll H^\nu(x) < t \rr.
\end{align}
\end{os}
\section{Hitting-times of the subordinated Poisson process}
In this section we study the hitting-times
\begin{align}
T_k^f \, = \, \inf \ll t\geq 0 : N^f(t) \geq k \rr,
\end{align}
of the subordinated Poisson processes.
The fact that $N^f(t)$ performs jumps of random height makes $T_k^f$ substantially different from the Erlang process related to the homogeneous Poisson process. Indeed, the law of $T_k^f$ can be written down as follows
\begin{align}
& \Pr \ll T_k^f \in ds \rr /ds \notag \\
 = \, & \Pr \ll \bigcup_{j=1}^k \ll N^f(s) = k-j , \, N^f[s, s+ds) \geq j \rr \rr /ds \notag \\
= \, &  \sum_{j=1}^k \Pr \ll N^f(s) = k-j \rr \sum_{m=j}^\infty \Pr \ll N^f[s, s+ds) = m \rr /ds \notag \\
= \, & \sum_{j=1}^k \int_0^\infty \Pr \ll N(z) = k-j \rr \Pr \ll H^f(s) \in dz \rr \sum_{m=j}^\infty \frac{\lambda^m}{m!} \int_0^\infty e^{-\lambda u} u^m \, \nu(du) \notag \\
= \, & \sum_{j=1}^k \int_0^\infty \frac{\l \lambda z \r^{k-j}}{(k-j)!} e^{-\lambda z} \Pr \ll H^f (s) \in dz \rr \int_0^\infty \Pr \ll N(u) \geq j \rr \nu(du) \notag \\
= \, & \sum_{j=1}^k \frac{(-\lambda)^{k-j}}{(k-j)!}  \int_0^\infty  \frac{d^{k-j}}{d\lambda^{k-j}} e^{-\lambda z} \Pr \ll H^f(s) \in dz \rr \int_0^\infty \Pr \ll N(u) \geq j  \rr \nu(du) \notag \\
= \, & \sum_{j=1}^k  \frac{(-\lambda)^{k-j}}{(k-j)!} \frac{d^{k-j}}{d\lambda^{k-j}} e^{-sf(\lambda)} \int_0^\infty \Pr \ll N(u) \geq j  \rr \nu(du) \notag \\
= \, & \sum_{l=0}^{k-1}  \frac{(-\lambda)^l}{l!} \l \frac{d^l}{d\lambda^l} e^{-sf(\lambda)} \r \int_0^\infty  \l 1-\sum_{r=0}^{k-l-1} \frac{(\lambda u)^r}{r!} e^{-\lambda u}  \r \nu(du).
\label{oggi}
\end{align}
The distribution of $T_k^f$ can also be obtained by observing that
\begin{align}
\Pr \ll T_k^f < s \rr \, = \, & \Pr \ll N^f(s) \geq k \rr \notag \\
= \, &\sum_{j=k}^\infty \int_0^\infty e^{-\lambda z} \frac{(\lambda z)^j}{j!} \Pr \ll H^f(s) \in dz \rr
\end{align}
and thus
\begin{align}
\Pr \ll T_k^f \in ds \rr \bigg/ ds \, = \, & \frac{d}{ds} \sum_{j=k}^\infty \int_0^\infty e^{-\lambda z} \frac{(\lambda z)^j}{j!} \Pr \ll H^f(s) \in dz \rr \notag \\
= \, & \frac{d}{ds} \int_0^\infty \Pr \ll N(z) \geq k \rr \Pr \ll H^f(s) \in dz \rr \notag \\
= \, & \frac{d}{d s } \int_0^\infty \l 1-\sum_{l=0}^{k-1} \frac{(\lambda z)^l}{l!} e^{-\lambda z} \r \Pr \ll H^f(s) \in dz \rr \notag \\
= \, & - \frac{d}{ds} \sum_{l=0}^{k-1} \frac{(-\lambda)^l}{l!} \int_0^\infty \frac{d^l}{d \lambda^l} e^{-\lambda z} \Pr \ll H^f(s) \in dz \rr \notag \\
= \, & -\frac{d}{ds} \sum_{l=0}^{k-1} \frac{(-\lambda)^l}{l!}  \frac{d^l}{d \lambda^l} e^{-sf(\lambda)}, \qquad s>0.
\label{mamma}
\end{align}
For $f(\lambda) = \lambda$ from \eqref{mamma}, we extract the Erlang distribution for the first-passage time of the Poisson process.
\begin{os}
In particular, we observe that from \eqref{oggi} and \eqref{mamma} we have that
\begin{align}
\Pr \ll T_1^f \in ds \rr \, = \, f(\lambda) e^{-sf(\lambda)} \, ds, \qquad s>0.
\label{first}
\end{align}
This proves that the waiting time of the first event for all subordinated Poisson processes is exponential. Instead
\begin{align}
\Pr \ll T_2^f \in ds \rr \, = \, e^{-sf(\lambda)} \l f(\lambda) -\lambda f^\prime (\lambda) + \lambda s f^\prime (\lambda) f(\lambda) \r \, ds, \qquad s>0,
\label{second}
\end{align}
and for $f(\lambda) = \lambda$ (ordinary Poisson case) we recover the gamma distribution with parameters $(2, \lambda)$. Result \eqref{second} can also be obtained from \eqref{oggi}. For $f(\lambda ) = \lambda^\alpha$  (space-fractional Poisson process) we have that
\begin{align}
\Pr \ll T_2^\alpha \in ds \rr \, = \, ds \lambda^\alpha e^{-s\lambda^\alpha} \l 1-\alpha + \alpha s \lambda^\alpha  \r, \qquad s>0.
\end{align}
Clearly, \eqref{second} cannot be the distribution of the sum of exponential r.v.'s \eqref{first} because the second event can also be obtained as a jump of magnitude equal to two. Finally we observe that
\begin{align}
\Pr \ll T_k^f \in ds \rr \, = \, \Pr \ll T_{k-1}^f \in ds \rr - \frac{(-\lambda)^{k-1}}{(k-1)!} \frac{d}{ds} \frac{d^{k-1}}{d\lambda^{k-1}} e^{-sf(\lambda)} ds, \, \qquad s \in (0, \infty),
\end{align}
so that the distributions of $T_k^f$ can be derived successively.
\end{os}
Here we derive the equation governing the distribution of $T_k^f$.
First we note that
\begin{align}
\mathfrak{G}^f(u, s) \, = \, \sum_{k=1}^\infty u^k \frac{\Pr \ll T_k^f \in ds \rr}{ds} \, = \, \frac{u}{1-u} f\l \lambda(1-u) \r e^{-sf(\lambda(1-u))}, \quad s>0, |u|<1.
\end{align}
This can be proved as follows
\begin{align}
\mathfrak{G}^f(u, s) \, = \,& \sum_{k=1}^\infty u^k \frac{\Pr \ll T_k^f \in ds \rr}{ds} \notag \\
= \, & \frac{d}{ds} \sum_{j=1}^\infty \sum_{k=1}^j u^k \int_0^\infty e^{-\lambda z} \frac{(\lambda z)^j}{j!} \Pr \ll H^f (s) \in dz \rr \notag \\
= \, & \frac{d}{ds} \sum_{j=1}^\infty \frac{u^{j+1}-u}{u-1} \int_0^\infty e^{-\lambda z} \frac{(\lambda z)^j}{j!} \Pr \ll H^f(s) \in dz \rr \notag \\
= \, & \frac{d}{ds} \int_0^\infty \frac{u}{u-1} \l e^{-\lambda z(1-u)} -1 \r \Pr \ll H^f(s) \in dz \rr \notag \\
= \, & \frac{u}{1-u} f(\lambda (1-u)) e^{-sf(\lambda(1-u))} 
\label{fretta}
\end{align}
\begin{te}
The probability density
\begin{align}
q_k^f(t)  \, = \, \Pr \ll T_k^f \in dt \rr \bigg/ dt
\end{align}
solves the equation
\begin{align}
f(\lambda (I-B)) q_k^f(t) \, = \, -\frac{d}{dt} q_k^f(t).
\end{align}
\end{te}
\begin{proof}
Since
\begin{align}
f(\lambda (I-B)) q_k^f(t) \, = \, f(\lambda) q_k^f(t) - \sum_{m=1}^{k-1} \int_0^\infty e^{-\lambda s} \frac{(\lambda s)^m}{m!} \nu(ds) q_{k-m}^f(t)
\end{align}
we can write, since $q_0(t) = 0$, for $t>0$,
\begin{align}
& \sum_{k=1}^\infty u^k f(\lambda (I-B)) q_k^f(t) \notag \\
 = \, & f(\lambda) \mathfrak{G}^f (u, t) - \sum_{k=1}^\infty u^k \sum_{m=1}^k \int_0^\infty e^{-\lambda s} \frac{(\lambda s)^m}{m!} \nu(ds) q_{k-m}^f(t) \notag \\
 = \, & f(\lambda) \mathfrak{G}^f (u, t) - \sum_{m=1}^\infty \sum_{k=m}^\infty \l \int_0^\infty e^{-\lambda s} \frac{(\lambda s)^m}{m!} \nu(ds) \r u^k q_{k-m}^f (t) \notag \\
 = \, & f(\lambda) \mathfrak{G}^f (u, t)- \mathfrak{G}^f(u, t) \int_0^\infty e^{-\lambda s} \l e^{u\lambda s} -1 \r \nu(ds) \notag \\
 = \, & f \l \lambda (1-u) \r \, \mathfrak{G}^f(u, t).
\end{align}
From \eqref{fretta} we get
\begin{align}
f \l \lambda (1-u) \r \mathfrak{G}^f(u, t) \, = \, -\frac{d}{dt} \mathfrak{G}^f(u, t),
\end{align}
which completes the proof.
\end{proof}

\section{Some particular cases}
In this section, we specialize the function $f$ in order to analyse some particular cases of $N^f(t)$, $t>0$.
\subsection{The space-fractional Poisson process}
If
\begin{align}
\nu(ds) \, = \, \frac{\alpha s^{-\alpha -1}}{\Gamma (1-\alpha)} ds, \qquad \alpha \in (0,1),
\end{align}
we obtain the space-fractional Poisson process $N^\alpha (t)$, $t>0$, studied in \citet{orspolspl}. The distributions of jumps \eqref{11} and \eqref{12} specialize to
\begin{align}
\Pr \ll N^\alpha [t, t+dt) = k \rr \, = \, &\begin{cases} \frac{(-1)^{k+1}\lambda^\alpha}{k!}  \alpha (\alpha -1) \cdots (\alpha - k +1) dt + o (dt), \quad &k > 0 \\
 1-\lambda^\alpha dt + o (dt), & k =0, \end{cases} \notag \\
 = \, & \lambda^\alpha  dt \frac{(-1)^{k+1} \Gamma (\alpha +1)}{k! \Gamma (\alpha +1-k)}.
\label{41}
\end{align}
since $f(\lambda) = \lambda^\alpha$.
The distribution of $N^\alpha$ can be written in three different ways as
\begin{align}
p_k^\alpha (t) \, = \, & \Pr \ll N^\alpha (t) = k \rr \, = \, \frac{(-1)^k}{k!} \sum_{r=0}^\infty \frac{(-\lambda^\alpha t)^r}{r!} \frac{\Gamma (\alpha r +1)}{\Gamma (\alpha r+1-k)} \notag \\
= \, &\frac{(-1)^k}{k!} \sum_{r=0}^\infty \frac{(-\lambda^\alpha t)^r}{r!} (\alpha r) (\alpha r -1) \cdots (\alpha r-k+1) \notag \\
= \, & \frac{(-1)^k}{k!} \frac{d^k}{du^k} e^{-t\lambda^\alpha u^\alpha} \bigg|_{u=1}
\label{42}
\end{align}
and we note that the probabilities \eqref{41} can be obtained directly from \eqref{42}.
\begin{os}
In light of \eqref{42} the distribution of the space-fractional Poisson process has the following alternative form
\begin{align}
p_k^\alpha (t) \, = \, \frac{e^{-\lambda^\alpha t}}{k!} \left[ c_{k, k}t^k + c_{k-1, k}t^{k-1} + \cdots + c_{2, k}t^2 + c_{1, k}t \right]
\label{43}
\end{align}
where the coefficients $c_{j, k}$, $j=1, \dots k,$ can be computed by means of successive derivatives.
In particular, we have that
\begin{align}
& c_{k,k} \, = \, \l \alpha \lambda^\alpha  \r^k, &c_{k-1, k} \, = \, \alpha^{k-1} \l 1-\alpha \r \frac{k(k-1)}{2} (\lambda^\alpha )^{k-1}, \notag  \\
& c_{2, k} \, = \, \l \lambda^\alpha  \r^2 \alpha^2 \prod_{j=1}^{k-2} \l j-\alpha \r \frac{k(k-1)}{2},  & c_{1, k} \, = \, \alpha \lambda^\alpha \prod_{j=1}^{k-1} (j-\alpha).
\label{44}
\end{align}
For $\alpha = 1$ all the coefficients $c_{j, k}$, $j=1, \dots, k-1,$ are equal to zero and we recover from \eqref{43} the distribution of the homogeneous Poisson process. The coefficients \eqref{44} are sufficient to obtain $p_j^\alpha (t)$, $1 \leq j \leq 4$ as
\begin{align}
\begin{cases}
p_2^\alpha (t) \, = \, &\frac{e^{-\lambda^\alpha t}}{2} \left[\l  \lambda^{\alpha} \alpha t \r^2  + \alpha (1-\alpha) \lambda^\alpha t \right] \\
p_3^\alpha (t) \, = \,& \frac{e^{-\lambda^\alpha t}}{3!} \left[ \l \lambda^\alpha \alpha t \r^3 + 3 \l \lambda^\alpha \alpha t \r^2 (1-\alpha) + \l \lambda^\alpha \alpha t \r \l 1-\alpha \r (2-\alpha) \right] \\
p_4^\alpha (t) \, = \, &\frac{e^{-\lambda^\alpha t}}{4!} \left[ \l \alpha \lambda^\alpha  t \r^4 + 6 \l \lambda^\alpha \alpha t \r^3 (1-\alpha) + 6 \l \alpha \lambda^\alpha t \r^2 (1-\alpha) (2-\alpha) \right. \\
&\left. + \lambda^\alpha \alpha t (1-\alpha)(2-\alpha) (3-\alpha) \right]
\end{cases}
\end{align}
\end{os}
\begin{os}
In light of the independence of increments for the space-fractional Poisson process we have that, for $0 \leq r \leq k$ and $0 \leq s \leq t$,
\begin{align}
\Pr \ll N^\alpha (s) = r | N^\alpha (t) = k \rr \, = \, & \frac{\Pr \ll N^\alpha (s) = r \rr \Pr \ll N^\alpha (t-s) = k-r \rr}{\Pr \ll N^\alpha (t) = k \rr} \notag \\
= \, & \left. \binom{k}{r} \frac{\frac{d^r}{du^r} e^{-s\lambda^\alpha u^\alpha} \frac{d^{k-r}}{du^{k-r}} e^{-(t-s)\lambda^\alpha u^\alpha}}{\frac{d^k}{du^k} e^{-\lambda^\alpha tu^\alpha}} \right|_{u=1} \notag \\
= \, & \binom{k}{r} \frac{\sum_{j=1}^r c_{j,r} s^j \; \sum_{n=1}^{k-r} c_{n, k-r}(t-s)^n}{\sum_{l=1}^k c_{l, k} t^l},
\label{nonserve}
\end{align}
where we used \eqref{43}.
For $\alpha =1$ we get that $c_{r,r}, c_{k-r, k-r}, c_{k, k} \neq 0$ and $c_{j,r} = c_{n, k-r} = c_{l, k} = 0$, for $j < r, \, n < k-r, \, l < k $ and thus we recover from \eqref{nonserve} the binomial distribution.
\end{os}
In the time interval $[0,t]$ the instants of occurences of the upward jumps are denoted by $\tau_j^{l_j}$, $1 \leq j \leq r$, $l_j \geq 1$, where $r$ is the number of jumps in $[0,t]$ and $l_j$ is the height of the $j$-th jump.
We can write the following distribution, for $r \leq k$,
\begin{align}
\Pr \ll \bigcap_{j=1}^r \ll \tau_j^{l_j} \in dt_j \rr \bigg| N^\alpha (t) = k \rr \, = \, \frac{k! \l \lambda^\alpha \Gamma(\alpha +1) \r^r (-1)^{k+r} \prod_{j=1}^r \frac{dt_j}{l_j! \Gamma (\alpha +1 - l_j)}}{\sum_{n=1}^k c_{n,k}t^n} 
\label{quaot}
\end{align}
for $0< t_1 < \cdots < t_r < t$. The distribution \eqref{quaot} can be evaluated by considering that
\begin{align}
& \Pr \ll \bigcap_{j=1}^r \ll \tau_j^{l_j} \in dt_j \rr \bigg| N^\alpha (t) = k \rr \notag \\
 = \, & \frac{1}{\Pr \ll N^\alpha (t) = k \rr} \Pr \ll \bigcap_{j=1}^{r+1} \ll N^\alpha [t_{j-1}, t_j) = 0, N^\alpha [t_j, t_j+dt_j) = l_j \rr \rr,
\end{align}
where $t_0 =0$ and $t_{r+1} = t$. Since the space-fractional Poisson process has independent increments and in view of the transition probabilities \eqref{41}, we arrive at \eqref{quaot}.
If $N^\alpha (t) = k$, and $l_j=1$, $\forall j$, we have that
\begin{align}
\Pr \ll \bigcap_{j=1}^k \ll \tau_j^1 \in dt_j \rr \bigg| N^\alpha (t) = k \rr \, = \, \frac{k! \l \alpha \lambda^\alpha \r^k}{\sum_{j=1}^k c_{j,k}t^j} \prod_{j=1}^k dt_j
 \label{412}
\end{align}
on the simplex
\begin{equation}
 S_t = \ll t_i, i =1, \dots, k : 0< t_1 < t_2 < \cdots < t_k < t \rr.
 \end{equation}
Clearly, for $\alpha =1$, we retrieve from \eqref{412} the uniform distribution on the set $S_t$.
Since the coefficients $c_{j,k}$ can be calculated in some specific cases, the distribution can be written down explicitly for small values of $k$. For example, for $k=2$ we have that
\begin{align}
&\Pr \ll \bigcap_{j=1}^2 \ll \tau_j^1 \in dt_j \rr \bigg| N^\alpha (t) = 2 \rr \, = \, \frac{\l \alpha \lambda^\alpha \r^2 \, dt_1dt_2}{\l \alpha \lambda^\alpha t \r^2 + \alpha (1-\alpha) \lambda^\alpha t}, \quad &0 < t_1 < t_2 < t, \\
&\Pr \ll \tau_1^2 \in dt_1 | N^\alpha (t) = 2 \rr = \frac{\frac{1}{2} \alpha (1-\alpha)\lambda^\alpha dt_1}{ \l \alpha \lambda^\alpha t \r^2 + \alpha (1-\alpha) \lambda^\alpha t}, & 0 < t_1 < t.
\end{align}

\subsection{Poisson process with a relativistic (tempered) stable subordinator}
In the case the L\'evy measure has the form
\begin{align}
\nu(ds) \, = \, \frac{\alpha s^{-\alpha -1}e^{-\theta s}}{\Gamma (1-\alpha)} ds, \qquad \theta > 0, 0<\alpha < 1,
\end{align}
we obtain an extension of the space-fractional Poisson process. This new Poisson process has the form $N^{\alpha, \theta}(t) \stackrel{\textrm{law}}{=} N \l H^{\alpha, \theta} (t) \r$, where $H^{\alpha, \theta}$ is the relativistic or tempered stable subordinator. Such a process is called relativistic since it appeared in the study of the stability of relativistic matter (see \citet{lieb}). From \eqref{212} we obtain the p.g.f. as
\begin{align}
G^{\alpha, \theta} (u, t) \, = \, &  \exp \ll -t \int_0^\infty \l 1-e^{-\lambda(1-u)s} \r \frac{\alpha s^{-\alpha -1 }e^{-\theta s}}{\Gamma (1-\alpha)}ds \rr \notag \\
= \, & e^{-t \ll \left[ \theta + \lambda(1-u) \right]^\alpha -\theta^\alpha \rr} \notag \\
= \, & e^{\theta^\alpha t} \sum_{k=0}^\infty \frac{\left[ -t \l \theta + \lambda(1-u) \r \right]^\alpha}{k!} \notag \\
= \, & e^{\theta^\alpha t} \sum_{k=0}^\infty \frac{\left[ -t \l \theta + \lambda \r^\alpha \right]^k}{k!} \l 1-\frac{\lambda u}{\theta + \lambda} \r^{\alpha k} \notag \\
= \, & e^{\theta^\alpha t} \sum_{k=0}^\infty \frac{\l -t (\theta + \lambda) \r^k}{k!} \sum_{m=0}^\infty \frac{\Gamma (\alpha k +1)}{\Gamma (\alpha k+1-m)m!} \l -\frac{\lambda u}{\theta + \lambda} \r^m \notag \\
= \, &\sum_{m=0}^\infty u^m \left[ \frac{(-1)^m}{m!} \frac{\lambda^m e^{\theta^\alpha  t}}{\l \theta + \lambda \r^m} \sum_{k=0}^\infty \frac{\left[ -t (\theta + \lambda)^\alpha \right]^k}{k!} \frac{\Gamma \l \alpha k +1 \r }{\Gamma \l \alpha k + 1 - m \r}\right].
\label{422}
\end{align}
From \eqref{422} we extract the distribution of $N^{\alpha, \theta}(t)$, $t>0$, as follows
\begin{align}
\Pr \ll N^{\alpha, \theta} (t) = m \rr \, = \, \frac{(-1)^m}{m!} \frac{\lambda^m e^{\theta^\alpha t}}{\l \theta + \lambda \r^m} \sum_{k=0}^\infty \frac{\l -t \l \lambda + \theta \r^\alpha \r^k}{k!} \frac{\Gamma (\alpha k +1)}{\Gamma (\alpha k + 1 -m)}, \quad m \geq 0.
\label{423}
\end{align}
For $\theta = 0$, formula \eqref{423} yields the distribution of the space-fractional Poisson process (see formula (1.2) of \citet{orspolspl}). An alternative form of \eqref{423} is
\begin{align}
\Pr \ll N^{\alpha, \theta} (t) = m \rr \, = \, \frac{(-1)^m}{m!} \l \frac{\lambda}{\lambda + \theta} \r^m e^{\theta^\alpha t} \frac{d^m}{du^m} e^{-tu^\alpha \l \theta + \lambda \r^\alpha} \bigg|_{u=1}
\label{424}
\end{align}
and can be derived either from \eqref{423} or from \eqref{219}. From \eqref{424} (and also from \eqref{11}) we have that, for $m \geq 1$,
\begin{align}
\Pr \ll N^{\alpha, \theta} [t, t+dt) = m \rr \, = \, \frac{(-1)^{m+1}\l \frac{\lambda}{\lambda + \theta} \r^m}{m!}  \l \lambda + \theta \r^\alpha \alpha (\alpha -1) \cdots (\alpha -m+1) dt
\label{420}
\end{align}
and this represents the distribution of the jumps during $[t, t+dt)$. Formula \eqref{420} shows that high jumps have less probability of occurring than in the space-fractional Poisson process.
\begin{os}
We notice that
\begin{align}
&\mathds{E} N^{\alpha, \theta} (t) \, = \, \alpha \lambda \theta^{\alpha -1}t, \notag \\
& \normalfont \textrm{Var} \left[ N^{\alpha, \theta} (t) \right] \, = \, \alpha \lambda  \theta^{\alpha -2} \l \lambda (1-\alpha) + \theta \r t, \notag \\
& \normalfont \textrm{Cov} \left[ N^{\alpha, \theta}(t) N^{\alpha, \theta}(s) \right] \, = \, \alpha \lambda  \theta^{\alpha -2} \l \lambda (1-\alpha) + \theta \r \, \l s \wedge  t \r.
\label{4222}
\end{align}
From \eqref{4222}, it is apparent that in the space-fractional Poisson process ($\theta = 0$) the mean values diverge.
\end{os}

\subsection{Poisson process with gamma subordinator}
For the L\'evy measure 
\begin{align} 
\nu(ds)=\frac{e^{-s}}{s}ds, \qquad s>0, 
\end{align}
the distribution of the related Poisson process has a particularly simple and interesting form, since it is the negative binomial. We note that the Bern\v{s}tein function corresponding to the L\'evy measure $\nu(ds) = \frac{e^{-s}}{s}ds$ is
\begin{align}
f(x) \, = \, \int_0^\infty \l 1-e^{-sx} \r \frac{e^{-s}}{s} ds \, = \, \log(1+x).
\end{align}
Therefore, the p.g.f. \eqref{212} reduces to the form
\begin{align}
G^\Gamma (u, t) \, = \, e^{-t\log(1+\lambda(1-u))} \, = \, \l 1+\lambda(1-u) \r^{-t},
\label{432}
\end{align}
and thus the intertime $T$ between successive clusters of events has law
\begin{align}
\Pr \ll T > t \rr \, = \, \frac{1}{(1+\lambda)^t}.
\end{align}
Formula \eqref{432} is clearly the p.g.f. of $N^\Gamma (t) \stackrel{\textrm{law}}{=} N \l H^\Gamma (t) \r$, where $H^\Gamma$ is the gamma subordinator with Laplace transform
\begin{align}
\mathds{E}e^{-\mu  H^f(t)} \, = \, \l 1+\mu \r^{-t}.
\end{align}
The distribution of $N^\Gamma (t)$, $t>0$, can be extracted from \eqref{432}.
\begin{prop}
The process $N^\Gamma (t)$, $t>0$, has the following distribution
\begin{align}
\Pr \ll N^\Gamma (t) = k \rr \, = \, & \begin{cases} \frac{\lambda^k t(t+1) \cdots (t+k-1)}{k!} \frac{1}{(\lambda +1)^{t+k}}, \qquad  &k \geq 1 \\ \frac{1}{(1+\lambda)^t}, & k=0, \end{cases} \notag \\
= \, & \frac{\lambda^k \Gamma(k+t)}{\Gamma(t) k! \l \lambda +1 \r^{t+k}}, \qquad k \geq 0.
\label{434}
\end{align}
\end{prop}
\begin{proof}
The distribution of $N \l H^\Gamma(t) \r$, $t>0$, is the negative binomial (see, for example, \citet{kozu}). For, its p.g.f. is
\begin{align}
G^\Gamma (u, t) \, = \, & \l 1+\lambda (1-u) \r^{-t} \notag \\
= \, & \l 1-\frac{\lambda u}{1+\lambda} \r^{-t} \l 1+\lambda \r^{-t} \, = \,  \l 1+\lambda \r^{-t} \sum_{k=0}^\infty \frac{\Gamma (-t+1)}{k! \Gamma (-t+1-k)} \l -\frac{\lambda u}{1+\lambda} \r^k \notag \\
= \, & \sum_{k=0}^\infty u^k \left[ \frac{\lambda^k \Gamma (t+k)}{k! \Gamma (t)}  \frac{1}{\l 1+\lambda \r^{t+k}} \right].
\end{align}
\end{proof}
\begin{os}
The distribution \eqref{434} of $N^\Gamma (t)$ is written as
\begin{align}
\Pr \ll N^\Gamma (t) = k \rr \, = \, & \frac{\lambda^k}{\l 1+\lambda \r^{k+t}} \frac{\Gamma (k+t)}{\Gamma (t)} \frac{1}{k!} \, = \, \mathbb{E} \Pr \ll N(\mathcal{T}) = k \rr,
\end{align}
where $\mathcal{T}$ is gamma distributed with parameters $(1,t)$ (that is the distribution of $H^\Gamma$) and $N$ is a homogeneous Poisson process with parameter $\lambda$, independent from $\mathcal{T}$. Furthermore, \eqref{434} can be regarded as an extension of the negative binomial $\mathcal{B}^i$ where
\begin{align}
\Pr \ll \mathcal{B}^i \, = \, k \rr \, = \, \frac{\Gamma(i+k)}{\Gamma(i)\Gamma(k+1)} p^i q^k
\end{align}
for $i=t$, $p=1/(1+\lambda)$, $q= \lambda/(1+\lambda)$ (see also \citet{kozu}).
\end{os}
\begin{coro}
The distribution of jumps in this case has the form
\begin{align}
\Pr \ll N^\Gamma [t, t+dt) = k \rr \, = \, \begin{cases} \l \frac{\lambda}{\lambda +1} \r^k \frac{1}{k}dt, \qquad & k \geq 1, \\ 1- \log(1+\lambda)dt, & k =0, \end{cases}
\label{430}
\end{align}
as can be inferred from \eqref{11} and also from \eqref{434}. The jumps follow logarithmic distribution.
\end{coro}
\begin{os}
We observe that, for $s<t$, $r\leq k$,
\begin{align}
\Pr \ll N^\Gamma (s) = r | N^\Gamma (t) = k \rr \, = \, &  \binom{k}{r} \frac{\Gamma (t)}{\Gamma (t-s)\Gamma (s)} \frac{\Gamma (s+r)\Gamma(t-s+k-r)}{\Gamma (k+t)} \notag \\
= \, & \binom{k}{r} \frac{B(s+r, t-s+k-r) }{B(s, t-s)}.
\label{beta}
\end{align}
Furthermore, from \eqref{beta} we can write, for $0 \leq r \leq k$,
\begin{align}
\Pr \ll N^\Gamma (s) = r | N^\Gamma (t) = k \rr \, = \, & \binom{k}{r} \frac{\int_0^1 x^{s+r-1} (1-x)^{t-s+k-r-1}dx}{B(s, t-s)} \notag \\
= \, & \mathds{E} \left[ \binom{k}{r}  X^{r} \l 1-X \r^{k-r} \right],
\label{stanza}
\end{align}
where $X$ is a r.v. with Beta distribution with parameter $s$ and $t-s$, that is
\begin{align}
\Pr \ll X \in dx \rr \, = \, \frac{x^{s-1}(1-x)^{t-s-1}}{B(s, t-s)} dx.
\end{align}
Formula \eqref{stanza} shows that in the gamma-Poisson process the conditional number of events at time $s<t$ is a randomized Bernoulli if $N(t) = k$.
\end{os}
\begin{os}
In view of \eqref{434}, \eqref{430}, and of the independence of the increments of the gamma-Poisson process, we have that
\begin{align}
\Pr \ll \bigcap_{j=1}^r \ll \tau_j^{l_j} \in dt_j \rr \bigg| N^\Gamma (t) = k \rr \, = \, \frac{k! \Gamma (t)}{\Gamma (t+k)} \prod_{j=1}^r \frac{dt_j}{l_j}
\label{fichissimo}
\end{align}
on the simplex $0<t_1 < t_2 < \cdots < t_r < t$ and $\sum_{j=1}^r l_j = k$.
Some special cases of \eqref{fichissimo} are
\begin{enumerate}
\item[i)] $l_j=1$, $\forall j = 1, \dots, r$, and thus $r=k$. In this case we have that
\begin{align}
\Pr \ll \bigcap_{j=1}^k \ll \tau_j^1 \in dt_1 \rr \bigg| N^\Gamma (t) = k \rr \, = \, \frac{k! \Gamma (t)}{\Gamma (t+k)} \prod_{j=1}^k dt_j, \qquad 0 < t_1 < \cdots < t_k < t;
\end{align}
\item[ii)] $l_1 = k$ and thus $r=1$ (unique jump of height $k$). Here we get
\begin{align}
\Pr \ll \tau_1^k \in dt_1 | N^\Gamma (t) = k \rr \, = \, \frac{dt_1}{k} \frac{k! \Gamma (t)}{\Gamma (t+k)}, \qquad 0 < t_1 < t;
\end{align}
\item[iii)] $k=2m$, $l_j=2$, $\forall j$, and therefore $r=m$. We have that
\begin{align}
\Pr \ll \bigcap_{j=1}^m \ll \tau_j^2 \in dt_j \rr \bigg| N^\Gamma (t) = 2m \rr \, = \, \frac{(2m)! \, \Gamma (t)}{2^m \Gamma (t+2m)} \prod_{j=1}^m dt_j, \end{align}
for $0 < t_1 < \cdots < t_m < t$.
\end{enumerate}
\end{os}
\begin{os}
From \eqref{432} we obtain the $r$-th factorial moment of $N^\Gamma(t)$, $t>0$, as
\begin{align}
\mathds{E} \left[ N^\Gamma (t) \l N^\Gamma (t) -1 \r \cdots \l N^\Gamma (t) -r+1 \r \right] \, = \, & \lambda^r t (t+1) \cdots (t+r-1).
\end{align}
While $\mathds{E}N^\Gamma (t) = \lambda t$, the variance becomes $\textrm{\normalfont Var} N^\Gamma (t) = \lambda t(\lambda +1)$ and
\begin{align}
\normalfont \textrm{Cov} \left[ N^\Gamma (t), N^\Gamma (s) \right] \, = \, \lambda (\lambda +1) (s \wedge t).
\end{align}
Furthermore, we have that
\begin{align}
&\mathds{E} \left[ \int_0^t N^\Gamma (s) ds \right] \, = \, \lambda t^2/2 \notag \\
& \normalfont \textrm{Var} \left[ \int_0^t N^\Gamma(s) ds \right] \, = \, \lambda (\lambda +1) t^3/3
\end{align}
\end{os}
\begin{os}
We can write also the following conditional mean values
\begin{align}
&\mathds{E} \left[ N^\Gamma(s) | N^\Gamma(t) = k \right] \, = \, \frac{ks}{t}, \qquad 0 < s < t, \\
& \mathds{E} \left[ N^\Gamma(s) N^\Gamma(w) | N^\Gamma(t) = k \right] \notag \\
 = \, & \frac{ks}{t}+ k (k-1) \frac{s(s+1)}{t(t+1)} + k(k-1) \frac{s(w-s)}{t(t+1)},\quad \textrm{\normalfont for }0<s<w<t \label{443} \\
 & \textrm{\normalfont Cov}\left[ N^\Gamma(s), N^\Gamma(w) | N^\Gamma(t) = k \right]  \notag   \\
  = \, & \frac{k}{t(t+1)} \l 1+\frac{k}{t} \r \, \min (s, w) \, \min(t-s, t-w).
 \label{444}
\end{align}
As a special case, we extract from \eqref{444} the conditional variance as
\begin{align}
\textrm{\normalfont Var} \left[ N^\Gamma(s) | N^\Gamma(t) = k \right] \, = \, \frac{sk(t-s)}{t(t+1)} \l 1+\frac{k}{t} \r, \qquad 0 < s < t,
\end{align}
and from \eqref{443}
\begin{align}
\mathds{E} \left[ \l N^\Gamma(s) \r^2 | N^\Gamma(t) = k \right] \, = \, \frac{s}{t} k + k(k-1) \frac{s}{t} \frac{s+1}{t+1}.
\end{align}
As a check, we observe that
\begin{align}
\textrm{\normalfont Var} N^\Gamma(s) \, = \,  & \mathds{E} \left[ \textrm{\normalfont Var} \left[ N^\Gamma(s) | N^\Gamma(t) \right] \right] + \textrm{\normalfont Var} \left[ \mathds{E} \left[ N^\Gamma(s) | N^\Gamma(t) \right] \right] \notag \\
= \, & \frac{s(t-s)}{t(t+1)} \mathds{E}N^\Gamma(t) + \frac{s(t-s)}{t^2(t+1)} \mathds{E}\l N^\Gamma (t) \r^2  + \frac{s^2}{t^2} \textrm{\normalfont Var} N^\Gamma(t) \notag \\
= \, & \frac{s}{t} \frac{t-s}{t+1} \lambda t + \frac{s}{t^2} \frac{t-s}{t+1} \l \lambda (\lambda +1)t+\lambda^2t^2 \r + \frac{s^2}{t^2} \lambda(\lambda +1) t \notag \\
= \, & \lambda (\lambda +1)s.
\end{align}
\end{os}
\begin{os}
We consider here the distribution of $N_1^\Gamma(t) - N_2^\Gamma(t)$, $t>0$, where $N_j^\Gamma$, $j=1, 2$, are independent gamma-Poisson processes. This leads to a generalization of the Skellam law of the difference of independent homogenous Poisson processes.  We have that
\begin{align}
& \Pr \ll N_1^\Gamma (t) - N_2^\Gamma (t) = r \rr \notag \\
 = \, & \sum_{k=0}^\infty \frac{\lambda^k \Gamma (k+t) \lambda^{k+r} \Gamma (k+r+t)}{(1+\lambda)^{k+t} \Gamma (t) k! (1+\lambda)^{k+r+t} (k+r)! \Gamma (t)}  \notag \\
 = \, & \frac{1}{(1+\lambda)^{2t} \Gamma^2 (t)} \sum_{k=0}^\infty \frac{\lambda^{2k+r}}{(1+\lambda)^{2k+r} k! (k+r)!} \int_0^\infty dw \int_0^\infty dz \, e^{-w-z} w^{k+t-1}z^{k+r+t-1} \notag \\
 = \, &\frac{1}{(1+\lambda)^{2t} \Gamma^2(t)}  \int_0^\infty \int_0^\infty dw \, dz \, e^{-w-z} w^{t-\frac{r}{2}-1} z^{\frac{r}{2}+t-1} \sum_{k=0}^\infty  \frac{\l \frac{\lambda}{1+\lambda}  \sqrt{wz} \r^{2k+r}}{k! (k+r)!} \notag \\
 = \, &\frac{1}{(1+\lambda)^{2t} \Gamma^2(t)} \int_0^\infty \int_0^\infty e^{-w-z} w^{t-\frac{r}{2}-1} z^{\frac{r}{2}+t-1}  I_r \l \frac{2\lambda \sqrt{wz}}{1+\lambda} \r   \notag \\
= \, & \int_0^\infty \int_0^\infty \Pr \ll N_1^u (1) - N_2^y(1) = r \rr e^{-\frac{u+y}{\lambda}} \frac{(uy)^{t-1} du dy}{\lambda^{2t} \Gamma^2(t)} \notag \\
= \, & \mathds{E} \Pr \ll N_1^U(1) -N_2^Y(1) = r \rr,
\end{align}
where $U$ and $Y$ are independent gamma r.v.'s with parameters $1$ and $t$, and $I_0(x)$ is a Bessel function.
For the reader's convenience, we recall that the Skellam distribution reads
\begin{align}
\Pr \ll N_1^\lambda (t) - N_2^\beta (t) = r \rr \, = \, e^{- \l \beta + \lambda \r t} \l \frac{\lambda}{\beta} \r^{\frac{r}{2}} I_{|r|} \l 2t\sqrt{\lambda \beta} \r, \qquad r \in \mathbb{Z},
\end{align}
for independent Poisson processes $N_1^\lambda$, $N_2^\beta$, with rate $\lambda$ and $\beta$, respectively.
\end{os}

\section*{Acknowledgement}
The authors are greatful to the referee for all suggestions which substantially improved the presentation of the paper.

\end{document}